\documentclass[12pt,draft]{article}

\usepackage{amsmath}
\usepackage{amssymb}
\usepackage{amsthm}
\usepackage{geometry}
\usepackage{titlesec}
\usepackage[numbers,sort&compress]{natbib}
\usepackage{tabularx}
\usepackage{booktabs}
\usepackage{mathtools}
\usepackage{multirow}
\usepackage{multicol}
\usepackage{threeparttable}
\usepackage{enumitem}

\numberwithin{equation}{section}
\newtheorem{theorem}{Theorem}[section]
\newtheorem{lemma}[theorem]{Lemma}
\newtheorem{proposition}[theorem]{Proposition}
\newtheorem{corollary}[theorem]{Corollary}
\theoremstyle{definition}
\newtheorem{definition}[theorem]{Definition}
\newtheorem{remark}[theorem]{Remark}
\newtheorem{example}[theorem]{Example}
\theoremstyle{plain}
\newtheorem{claim}{Claim}

\titleformat*{\section}{\large\bfseries}
\titleformat*{\subsection}{\normalsize\bfseries}

\newcommand{\bC}{\ensuremath{\mathbb{C}}}

\newcommand{\bN}{\ensuremath{\mathbb{N}}}
\newcommand{\bZ}{\ensuremath{\mathbb{Z}}}
\newcommand{\subno}[1]{\noindent\textnormal{(#1)}}

\newcommand{\supp}{\mathrm{supp }}

\newcommand{\Lie}{\mathrm{Lie }}
\newcommand{\bZz}{\ensuremath{\mathbb{Z}_2}}
\newcommand{\ep}{\ensuremath{{\bar{0}}}}
\newcommand{\op}{\ensuremath{{\bar{1}}}}
\newcommand{\cp}{\ensuremath{\mathbb{C}[\partial]}}

\newcommand{\Ind}{\mathrm{Ind }}

\newcommand{\fg}{\ensuremath{\mathfrak{g}}}
\newcommand{\cS}{\mathcal{S}}
\newcommand{\cT}{\mathcal{T}}

\newcommand{\cU}{\mathcal{U}}
\newcommand{\fq}{\ensuremath{\mathfrak{q}}}

\newcommand{\bM}{\ensuremath{\mathbb{M}}}
\newcommand{\bi}{\mathbf{i}}
\newcommand{\bj}{\mathbf{j}}
\newcommand{\bk}{\mathbf{k}}
\newcommand{\bx}{\mathbf{x}}
\newcommand{\by}{\mathbf{y}}
\newcommand{\bz}{\mathbf{z}}
\newcommand{\bw}{\mathbf{w}}
\newcommand{\bd}{\mathbf{d}}
\newcommand{\bO}{\mathbf{0}}
\newcommand{\bY}{\ensuremath{\mathbb{Y}}}
\newcommand{\vijk}{\ensuremath{v}_{\mathbf{i,j,k}}}
\newcommand{\vxyz}{\ensuremath{v}_{\mathbf{x,y,z}}}
\newcommand{\hk}{\hat{k}}
\newcommand{\hj}{\hat{j}}
\newcommand{\ci}{\check{i}}
\newcommand{\vxy}{\ensuremath{v}_{\mathbf{x,y,0}}}
\newcommand{\vx}{\ensuremath{v}_{\mathbf{x,0,0}}}
\newcommand{\ti}{\tilde{i}}
\newcommand{\tj}{\tilde{j}}
\newcommand{\tk}{\tilde{k}}

\begin{document}
\begin{center}
    \bfseries\large
		Simple restricted modules over a new Lie superalgebra extended by the Ovsienko--Roger algebra
		\footnote{Supported by the National Natural Science Foundation of China (No. 12471027) and Natural Science Foundation of Shanghai (No. 24ZR1471900).
		\\ \indent \ \ $^*$Corresponding author: Xiaoqing Yue (xiaoqingyue@tongji.edu.cn)}

    \mdseries\normalsize
    \bigskip
		Jinrong Wang, Xiaoqing Yue$^*$
		
    \footnotesize
    \smallskip
		School of Mathematical Sciences, Key Laboratory of Intelligent Computing and  Applications (Ministry of Education), Tongji University, Shanghai 200092, China
		
    \smallskip
		E-mails: 2130926@tongji.edu.cn, xiaoqingyue@tongji.edu.cn 
\end{center}
    {
    \footnotesize
    \noindent\textbf{Abstract}. 
        In this paper, we introduce a new infinite-dimensional Lie superalgebra  
    $\cS$ called the super extended Ovsienko--Roger algebra. This algebra is obtained by determining the annihilation superalgebra of the Lie conformal superalgebra $S=S_{\bar0}\oplus S_{\bar{1}}$ with $S_{\bar{0}}=\mathbb{C}[\partial]L\oplus\mathbb{C}[\partial]W$, $S_{\bar{1}}=\mathbb{C}[\partial]G$ and non-trivial $\lambda$-brackets $[L_\lambda L]=(\partial+2\lambda)L$, $[L_\lambda G]=(\partial+\lambda)G$, $[L_\lambda W]=[G_\lambda G]=\partial W$. Then we construct a class of simple restricted $\cS$-modules, which are induced from simple modules of some finite dimensional solvable Lie superalgebras under certain conditions. Moreover, we obtain the classification of simple generalized Verma modules over $\cS$ and we show that the Verma module of $\cS$ is always reducible.
		
		
    }

\section{Introduction}\label{sect:intro}
\hspace{1.5em}Throughout this paper, we denote by $\bC,\bZ,\bN$ and $\bZ_+$ the sets of complex numbers, integers, nonnegative integers and positive integers, respectively. All vector spaces, Lie superalgebras and Lie conformal superalgebras are over $\bC$. For a Lie algebra $\fg$, we denote by $\cU(\fg)$ the universal enveloping algebra of $\fg$. We use $V[\lambda]$ to denote the set of polynomials of $\lambda$ with coefficients in the vector space $V$. 

The study of restricted modules (or smooth modules cf. \cite{KL}) is important for the representation theory of $\bZ$ or $\frac{1}{2}\bZ$-graded Lie superalgebras. It has been shown that there exists an equivalence between the category of restricted modules of Lie superalgebras and the category of modules of their corresponding vertex superalgebras (cf. \cite{FZ,LL,LPX22,LPXZ}). Mazorchuk and Zhao in \cite{MZ} proposed a general construction of simple Virasoro modules which includes and generalizes both highest weight modules and Whittaker modules. Since then, simple restricted modules over many known Lie (super)algebras are constructed and characterized (cf. \cite{C,CHS,LPX22,LPXZ,CSX}).

The concept of Lie conformal (super)algebras was introduced by Kac in \cite{K98}, which encodes an axiomatic description of the operator product expansion of chiral fields in conformal field theory. Moreover, the theory of Lie conformal superalgebras gives us an adequate tool for the study of Lie superalgebras satisfying the locality property \cite{K97}. In the past two decades, the structure and representation theories of Lie conformal superalgerbras have been greatly developed (cf. \cite{CK,CKW,FK,FKR,K98}). 

In the process of classifying Lie conformal superalgebras of rank $(2+1)$ \cite{WY}, the authors constructed several new and interesting examples. Their corresponding infinite-dimensional Lie superalgebras were also determined via the construction of their annihilation superalgebras. In particular, one of these new Lie superalgebras can be regarded as a super extension of the $\lambda=1$ Ovsienko--Roger algebra. The Ovsienko--Roger algebra $\mathcal{L}_\lambda$ is an extension of the Virasoro algebra of its tensor density module of degree one. It arose in the work of Ovsienko and Roger \cite{OR96} on matrix analogues of the Sturm--Liouville operators \cite{MOG}. The algebra $\mathcal{L}_0$ is the twisted Heisenberg--Virasoro algebra, which plays an important role in moduli spaces of curves \cite{ADKP}. The algebra $\mathcal{L}_{-1}$ is the $W(2,2)$ algebra, or known as the $\mathfrak{bms}_3$ algebra, which is widely studied in the context of vertex operator algebras and $\mathrm{BMS/GCA}$ correspondence (cf. \cite{B,ZD}). The algebra $\mathcal{L}_1$ is a non-perfect Lie algebra and the category of its restricted modules was studied in \cite{LPX22}.

Motivated by these advances, we introduce and study a new infinite-dimensional Lie superalgebra which we call super extended Ovsienko--Roger algebra.
\begin{definition}\label{def:e}
    The \emph{super extended Ovsienko--Roger algebra} $\mathcal{S}$ is a Lie superalgebra defined over $\bC$  with basis $$\{L_m, W_m, G_r, C_1, C_2\ |\ m\in\bZ,\ r\in\bZ+\frac{1}{2}\}$$ and super-brackets:
    \begin{equation}\label{e.rel}
        \begin{split}
            &[L_m,L_n]=(n-m)L_{m+n}+\frac{m^3-m}{12}\delta_{m+n,0}C_1,\\
            &[L_m,W_n]=(m+n)W_{m+n}+\delta_{m+n,0}C_2,\\
            &[L_m,G_r]=rG_{m+r},\\
            &[G_r,G_s]=(r+s)W_{r+s}+\delta_{r+s,0}C_2,\\
            &[W_m,W_n]=[W_m,G_r]=[\mathcal{S},C_1]=[\mathcal{S},C_2]=0,
        \end{split}
    \end{equation}
    for any $m,n\in\bZ$, $r,s\in\bZ+\frac{1}{2}$, where $|L_m|=|W_m|=|C_1|=|C_2|=\ep$ and $|G_r|=\op$.
\end{definition}

Recall that a special rank $(2+1)$ Lie conformal superalgebra was introduced in \cite{WY}.
\begin{definition}\label{def:exc.Lie.con.salg}
    The special rank $(2+1)$ Lie conformal superalgebra $S=S_\ep\oplus S_\op$ is a free $\bZz$-graded $\cp$-module generated by $L, W$ and $G$ satisfying the following non-trivial $\lambda$-brackets:
    \begin{align}\label{exc.lie.con.salg}
        [L_\lambda L]=(\partial+2\lambda)L,\ \ [L_\lambda W]=\partial W,\ \ [L_\lambda G]=(\partial+\lambda)G,\ \ [G_\lambda G]=\partial W,
    \end{align}
    where $S_\ep=\cp L\oplus\cp W$ and $S_\op=\cp G$.
\end{definition}
First, by determining the annihilation superalgebra of $S$, we construct a new class of Lie superalgebras $\bar{\cS}^\epsilon$, with $\epsilon=0$ and $\epsilon=\frac{1}{2}$ corresponding to the Ramond and Neveu-Schwarz types, respectively. We then compute its central extension in detail. However, we note that this is not the universal central extension, due to the fact that the Lie superalgebra $\bar{\cS}^\epsilon$ is non-perfect. 

With these preparations in place, we obtain the Lie superalgebra $\mathcal{S}$ and then turn our focus to the study of its simple restricted modules. We construct (Theorem~\ref{thm:simple.restricted.module}) and classify simple restricted $\cS$-modules under certain conditions (Theorem~\ref{thm:isom}). We also establish a one-to-one correspondence between simple restricted $\cS$-modules and simple modules of a family of finite-dimensional solvable Lie superalgebras associated with $\cS$ (Theorem~\ref{thm:quo.alg.mod}). Using Block's conclusions in \cite{Bl}, we obtain the classification of simple generalized Verma modules over $\cS$ (Corollary~\ref{cor:gen.ver.mod}). In addition, inspired by the condition in Theorem~\ref{thm:simple.restricted.module}, we show that the Verma module over $\cS$ is always reducible.

This paper is organized as follows. In Section~\ref{sect:prelim}, we review some notations and preliminaries for Lie superalgebras and Lie conformal superalgebras. In Section~\ref{sect:lie.salg.e}, we introduce new super extended Ovsienko--Roger algebras from the view of annihilation superalgebras and central extension. In Section~\ref{sect:construction}, a class of simple restricted modules over $\cS$ are constructed. In Section~\ref{sect:char}, we give several equivalent characterizations of simple restricted $\cS$-modules. Furthermore, simple generalized Verma modules for $\cS$ are classified. At last, we present some examples of restricted $\cS$-modules.

\section{Preliminaries}\label{sect:prelim}
\hspace{1.5em}In this section, we will make some preparations of Lie superalgebras and Lie conformal superalgebras for later use.

\subsection{Lie superalgebra}

\hspace{1.5em}Let $V$ be a superspace that is a \bZz-graded linear space with a direct sum $V=V_\ep \oplus V_\op$. If $x\in V_\theta$, $\theta\in\bZz=\{\ep,\op\}$, then we say $x$ is \emph{homogeneous} and of \emph{parity $\theta$}. Define $x\in V_\ep$ is \emph{even} and $x\in V_\op$ is \emph{odd}. The parity of a homogeneous element $x$ is denoted by $|x|$. Throughout what follows, if $|x|$ occurs in an expression, then it is assumed that $x$ is homogeneous and that the expression extends to the other elements by linearity. 

Let $\mathfrak{g}=\fg_\ep\oplus\fg_\op$ be a Lie superalgebra. A $\mathfrak{g}$-\emph{module} is a $\bZz$-graded vector space $V=V_\ep\oplus V_\op$ together with a bilinear map $\mathfrak{g}\times V\rightarrow V$, denoted by $(x,v)\mapsto xv$ such that
\begin{align*}
    x(yv)-(-1)^{|x||y|}y(xv)=[x,y]v,\ \ \mathfrak{g}_{\bar{i}}V_{\bar{j}}\subseteq V_{\bar{i}+\bar{j}},
\end{align*}
for all $x,y\in\mathfrak{g}$, $v\in V$, $\bar{i},\bar{j}\in\bZz$. Hence, there is a parity-change functor $\Pi$ on the category of $\mathfrak{g}$-modules to itself. In other words, for any module $V=V_\ep\oplus V_\op$, we have a new module $\Pi(V)$ with the same underlining space under the parity change, namely $\Pi(V_\ep)=V_\op$ and $\Pi(V_\op)=V_\ep$. Throughout this paper, all modules considered in this paper are $\bZz$-graded and all simple modules are non-trivial unless specified.

\begin{definition}
    Let $V$ be a module of a Lie superalgebra $\fg$ and $x\in\fg$.

    \subno{1} If for any $v\in V$ there exists $n\in\bN$ such that $x^nv=0$, then we call that the action of $x$ on $V$ is \emph{locally nilpotent}. Similarly, the action of $\fg$ on $V$ is \emph{locally nilpotent} if for any $v\in V$ there exists $n\in\bN$ such that $\fg^n v=0$.

    \subno{2} If for any $v\in V$ we have $\dim(\sum_{n\in\bN}\bC x^nv)<+\infty$, then we call that the action of $x$ on $V$ is \emph{locally finite}. Similarly, the action of $\fg$ on $V$ is \emph{locally finite} if for any $v\in V$ we have $\dim(\sum_{n\in\bN}\bC \fg^nv)<+\infty$.
\end{definition}

\begin{remark}
    It is easy to see that the action of $x$ on $V$ is locally nilpotent implies that the action of $x$ on $V$ is locally finite. If $\fg$ is a finitely generated Lie superalgebra, then the action of $\fg$ on $V$ is locally nilpotent implies that the action of $\fg$ on $V$ is locally finite. 
\end{remark}

\subsection{Lie conformal superalgebra}
\hspace{1.5em}Now, we recall some definitions, notations and related results about the Lie conformal superalgebra. One can refer to \cite{CK,FKR,K97,K98} for more details.
\begin{definition}\label{def:lie conformal superalgebra}
    A \emph{Lie conformal superalgebra} $R=R_\ep \oplus R_\op$ is a \bZz-graded $\cp$-module with a $\bC$-linear map $R\otimes R\rightarrow \bC[\lambda]\otimes R$, $a\otimes b\mapsto [a_\lambda b]$ called $\lambda$-bracket, and satisfying the following axioms $(a,b,c\in R)$:
    \begin{align*}
        \text{(conformal sesquilinearity)}&\phantom{1234} [\partial a_\lambda b]=-\lambda[a_\lambda b],\ [a_\lambda \partial b]=(\partial+\lambda)[a_\lambda b],\\
        \text{(super skew-symmetry)}&\phantom{1234} [a_\lambda b]=-(-1)^{|a||b|}[b_{-\lambda-\partial}a],\\
        \text{(super Jacobi identity)}&\phantom{1234} [a_\lambda[b_\mu c]]=[[a_\lambda b]_{\lambda+\mu}c]+(-1)^{|a||b|}[b_\mu[a_\lambda c]].
    \end{align*}
\end{definition}

For a Lie conformal superalgebra $R=R_\ep\oplus R_\op$, $R$ is called \emph{finite} if it is finitely generated over $\bC[\partial]$. If it is further a free $\bC[\partial]$-module, then the \emph{rank} of $R$ is defined as the rank of the underlying free $\bC[\partial]$-module. Sometimes the rank is written as the sum of two numbers which means the sum of the rank of $R_\ep$ and the rank of $R_\op$.

In particular, there exists an important Lie superalgebra associated with a Lie conformal superalgebra. Suppose $R$ is a Lie conformal superalgebra. For any $a,b\in R$, we write 
\begin{align*}
    [a_\lambda b]=\sum_{j\in\bN}\lambda^{(j)}a_{(j)}b\ \text{ with }\ \lambda^{(j)}=\frac{\lambda^j}{j!}.
\end{align*}
For every $j\in\bN$, we have the \bC-linear map: $R\otimes R\rightarrow R,\ a\otimes b\mapsto a_{(j)}b$, which is called the $j$-\emph{th product}. 

\begin{definition}\label{def:annalg}
    The \textit{annihilation superalgebra} $\Lie(R)$ of a Lie conformal superalgebra $R$ is a Lie superalgebra spanned over $\bC$ by $\{a_{(n)}\ |\ a\in R,\ n\in\bN\}$ with relations
    \begin{align*}
		(\partial a)_{(n)}=-na_{(n-1)},\ \ (a+b)_{(n)}=a_{(n)}+b_{(n)},\ \ (ka)_{(n)}=ka_{(n)},
    \end{align*}
    for $a,b\in R$ and $k\in\bC$, and the Lie brackets of $\Lie(R)$ are given by
    \begin{equation}\label{ann.lie.bracket}
		[a_{(m)},b_{(n)}]=\sum_{j\in\bN}\binom{m}{j}(a_{(j)}b)_{(m+n-j)}.
    \end{equation}
    The parity $|a_{(n)}|$ of $a_{(n)}\in\Lie(R)$ is the same as $|a|$ for any homogeneous $a\in R$ and $n\in\bN$.
\end{definition}

For a Lie conformal superalgebra $R$, the annihilation superalgebra $\Lie(R)$ and the \textit{extended annihilation superalgebra} $\Lie(R)^e=\bC\partial\ltimes\Lie(R)$ with $[\partial,a_{(n)}]=-na_{(n-1)}$ play crucial roles in the representations of Lie conformal superalgebras (cf. \cite{CK}).

\begin{example}
    The Neveu--Schwarz Lie conformal superalgebra is a free \bZz-graded $\bC[\partial]$-module with a basis $\{L,\, G\}$ satisfying the following relations:
    \begin{align*}
    [L_\lambda L]=(\partial+2\lambda)L,\ \ [L_\lambda G]=(\partial+\frac32\lambda)G,\ \ [G_\lambda G]=2L.
\end{align*}
   It is an important example of the Lie conformal superalgebra introduced in \cite{CK} firstly, where its finite nontrivial irreducible conformal modules were studied. Actually, the Neveu--Schwarz Lie conformal superalgebra corresponds to the super-Virasoro algebra (see \cite{CKW} for details), also called the $N=1$ superconformal algebra (cf. \cite{BMRW,K98,NS,Ra}). The (centreless) super-Virasoro algebras are the Lie superalgebras $\mathrm{SVir}_\epsilon=\bigoplus_{m\in\bZ}\bC L_m\oplus \bigoplus_{r\in\bZ+\epsilon}\bC G_r$, which satisfy the following relations:
   \begin{align*}
       [L_m,L_n]=(m-n)L_{m+n},\ [L_m,G_r]=(\frac{m}{2}-r)G_{m+r},\ [G_r,G_s]=2L_{r+s},
   \end{align*}
   for all $m,n\in\bZ$, $\epsilon=\frac{1}{2},0$, $r,s\in\bZ+\epsilon$. $\mathrm{SVir}_0$ is called the Ramond algebra and $\mathrm{SVir}_\frac{1}{2}$ is called the Neveu--Schwarz algebra. The representation theory for these Lie superalgebras are studied in a series of papers such as \cite{CLL,IK1,S,C,LPX20}.
\end{example}

\section{New Lie superalgebras derived from $S$}\label{sect:lie.salg.e}

\subsection{Annihilation superalgebra $\Lie(S)$}
\begin{lemma}\label{lem:bar.e}
    Suppose that the Lie superalgebra $\bar{\cS}^0$ is the annihilation superalgebra of $S$. Then $\bar{\cS}^0=\Lie(S)=\bigoplus_{m\in\bZ}(\bC L_m\oplus\bC W_m\oplus\bC G_m)$ satisfies the following relations\textup{:}
    \begin{align*}
        &[L_m,L_n]=(n-m)L_{m+n},\ \ [L_m,W_n]=(m+n)W_{m+n},\ \ [L_m,G_n]=nG_{m+n},\\
        &[G_m,G_n]=(m+n)W_{m+n},\ \ [W_m,W_n]=[W_m,G_n]=0,
    \end{align*}
    where $|L_m|=|W_m|=\ep$ and $|G_m|=\op$, $m,n\in\bZ$.
\end{lemma}
\begin{proof}
    By the definition of the $j$-th product and \eqref{exc.lie.con.salg}, we have
    \begin{align*}
        &L_{(0)}L=\partial L,\ L_{(1)}L=2L,\ L_{(0)}W=\partial W,\ L_{(0)}G=\partial G,\ L_{(1)}G=G,\ G_{(0)}G=\partial W,\\
        &L_{(i+2)}L=L_{(i+1)}W=L_{(i+2)}G=G_{(i+1)}G=W_{(i)}W=W_{(i)}G=0,\ \ \forall\ i\ge0.
    \end{align*}
    This together with \eqref{ann.lie.bracket} shows
    \begin{align*}
        &[L_{(m)},L_{(n)}]=(m-n)L_{(m+n-1)},\ \ [L_{(m)},W_{(n)}]=-(m+n)W_{(m+n-1)},\\
        &[L_{(m)},G_{(n)}]=-nG_{(m+n-1)},\ \ [G_{(m)},G_{(n)}]=-(m+n)W_{(m+n-1)},
    \end{align*}
    and the other items vanish. Shifting $L_{(m+1)}\rightarrow L_m$, $W_{(m-1)}\rightarrow W_m$, $G_{(m)}\rightarrow G_m$ and then taking $m\rightarrow-m$ for all $m\in\bZ$, we get the conclusion in Lemma~\ref{lem:bar.e}.
\end{proof}

Referring to the super-Virasoro algebra, perhaps we can call the Lie superalgebra $\bar{\cS}^0$ is of Ramond type. Meanwhile, we can define the Neveu--Schwarz type denoted by $\bar{\cS}^\frac{1}{2}$ with the basis $\{L_m, W_m, G_r\ |\ m\in\bZ,\ r\in\bZ+\frac{1}{2}\}$ satisfying 
\begin{align*}
    &[L_m,L_n]=(n-m)L_{m+n},\ \ [L_m,W_n]=(m+n)W_{m+n},\ \ [L_m,G_r]=rG_{m+r},\\
    &[G_r,G_s]=(r+s)W_{r+s},\ \ [W_m,W_n]=[W_m,G_r]=0,
\end{align*}
where $m,n\in\bZ$, $r,s\in\bZ+\frac{1}{2}$. Obviously, $\bar{\cS}^\frac{1}{2}$ is a subalgebra of $\bar{\cS}^0$. Let $\varphi:\bar{\cS}^\frac{1}{2}\rightarrow\bar{\cS}^0$ be a linear map defined by
\begin{align*}
    L_m\mapsto\frac{1}{2}L_{2m},\ \ W_{m}\mapsto W_{2m},\ \ G_{m+\frac{1}{2}}\mapsto \frac{1}{\sqrt{2}}G_{2m+1},\ \ \text{for }m\in\bZ.
\end{align*}
It is easy to see that $\varphi$ is injective. As a consequence, we obtain the following Lie superalgebras $\bar{\cS}^\epsilon$ derived from the Lie conformal superalgebra $S$.
\begin{definition}
    The Lie superalgebras 
    \begin{equation*}
        \bar{\cS}^\epsilon=(\bar{\cS}^\epsilon)_\ep\oplus(\bar{\cS}^\epsilon)_\op=\bigoplus_{m\in\bZ}(\bC L_m\oplus\bC W_m)\oplus\bigoplus_{r\in\bZ+\epsilon}\bC G_r
    \end{equation*}
    are defined over $\bC$, where $(\bar{\cS}^\epsilon)_\ep=\bigoplus_{m\in\bZ}(\bC L_m\oplus\bC W_m)$ and $(\bar{\cS}^\epsilon)_\op=\bigoplus_{r\in\bZ+\epsilon}\bC G_r$, with the following super-brackets:
\begin{align*}
    &[L_m,L_n]=(n-m)L_{m+n},\ \ [L_m,W_n]=(m+n)W_{m+n},\ \ [L_m,G_r]=rG_{m+r},\\
    &[G_r,G_s]=(r+s)W_{r+s},\ \ [W_m,W_n]=[W_m,G_r]=0,
\end{align*}
for any $m,n\in\bZ$, $r,s\in\bZ+\epsilon$, $\epsilon=\frac{1}{2},0$. 
\end{definition}

\subsection{Central extension}
\hspace{1.5em}Let $\fg$ be a Lie superalgebra. A 2-\emph{cocycle} on $\fg$ is a $\bC$-bilinear function $\psi: \fg\times\fg\rightarrow\bC$ satisfying $(x,y,z\in\fg)$:
\begin{align}
     \text{(super skew-symmetry)}&\phantom{12} \psi(x,y)=-(-1)^{|x||y|}\psi(y,x),\label{2-cocycle.skew.symm}\\
        \text{(super Jacobi identity)}&\phantom{12} \psi(x,[y,z])=\psi([x,y],z])+(-1)^{|x||y|}\psi(y,[x,z]).\label{2-cocycle.Jacobi}
\end{align}
Denote by $C^2(\fg,\bC)$ the vector space of 2-cocycles on $\fg$. For any linear function $f:\fg\rightarrow\bC$, we can define a 2-cocycle $\psi_f$ on $\fg$ by
\begin{equation*}
    \psi_f(x,y)=f([x,y]),\ \ \ \forall\ x,y\in\fg,
\end{equation*}
which is called a 2-\emph{coboundary} on \fg. Denote by $B^2(\fg,\bC)$ the vector space of 2-coboundaries on \fg. A 2-cocycle $\phi$ is said to be \emph{equivalent} to another 2-cocycle $\psi$ if $\phi-\psi$ is a 2-coboundary. The quotient space
\begin{equation*}
    H^2(\fg,\bC)=C^2(\fg,\bC)/B^2(\fg,\bC)
\end{equation*}
is called the \emph{second cohomology group} of $\fg$ with trivial coefficients in \bC.

\begin{definition}
    Let $\fg=\fg_\ep\oplus\fg_\op$ be a Lie superalgebra. The pair $(\tilde{\fg},\varphi)$, where $\tilde{\fg}=\tilde{\fg}_\ep\oplus\tilde{\fg}_\op$ is a Lie superalgebra and $\varphi:\tilde{\fg}\rightarrow\fg$ is an epimorphism, is called a \emph{central extension} of $\fg$ by $\ker(\varphi)$ if $[\ker(\varphi),\tilde{\fg}]=0$. Regard $\ker(\varphi)$ as a super-commutative Lie superalgebra, and at the same time, a trivial \fg-module. We say that the central extension is \emph{even} (resp. \emph{odd}) if $\ker(\varphi)_\op=\{0\}$ (resp. $\ker(\varphi)_\ep=\{0\}$).
\end{definition}

Let $\fg$ be a Lie superalgebra. A one-dimensional central extension $(\tilde{\fg},\varphi)$ of $\fg$ by the central element $z$ is an exact sequence of Lie superalgebra homomorphisms
\begin{equation*}
    0\xrightarrow{\phantom{12345}}\bC z\xrightarrow{\phantom{12345}}\tilde{\fg}\xrightarrow{\phantom{12}\varphi\phantom{12}}\fg\xrightarrow{\phantom{12345}}0,
\end{equation*}
where $\tilde{\fg}\cong\fg\oplus\bC z$ as vector space. Then the exact sequence is equivalent to the existence of a bilinear form $\alpha$, such that
\begin{equation*}
    [x,y]_{\tilde{\fg}}=[x,y]_{\fg}+\alpha(x,y)z,\ \ \forall\ x,y\in\fg.
\end{equation*}
It is straightforward to check that $\alpha$ is a 2-cocycle on $\fg$. Furthermore, we have
\begin{proposition}\label{prop:2.coho}\textup{(\cite{N})}
    The second cohomology group $H^2(\fg,\bC)$ is in one-to-one correspondence with the set of the equivalence classes of the one-dimensional central extension of $\fg$.
\end{proposition}

Let $\alpha$ be any 2-cocycle on $\bar{\cS}^\epsilon$. Assume that $f:\bar{\cS}^\epsilon\rightarrow\bC$ is a linear map over $\bar{\cS}^\epsilon$ defined by
\begin{align*}
    &f(L_n)=\frac{1}{n}\alpha(L_0,L_n),\ \ f(W_n)=\frac{1}{n}\alpha(L_0,W_n),\ \ n\neq 0,\\
    &f(G_r)=\frac{1}{r}\alpha(L_0,G_r),\ \ r\neq 0.
\end{align*}
Set $\bar{\alpha}=\alpha-\alpha_f$, where $\alpha_f(x,y)=f([x,y])$ for any $x,y\in\bar{\cS}^\epsilon$.

Since $\bar{\cS}^\epsilon_\ep$ is a subalgebra of $\mathcal{W}(a,b)$, by \cite[Theorem~2.3]{GJP}, we have the following lemma:
\begin{lemma}\textup{(\cite{GJP})}\label{lem:CE.1}
    For some $c_1,c_2,c_3\in\bC$, we have 
    \begin{align*}
        &\bar{\alpha}(L_m,L_n)=\frac{m^3-m}{12}\delta_{m+n,0}c_1,\\
        &\bar{\alpha}(L_m,W_n)=m\delta_{m+n,0}c_2+\delta_{m+n,0}c_3,\\
        &\bar{\alpha}(W_m,W_n)=0,\ \ \forall\ m,n\in\bZ.
    \end{align*}
\end{lemma}

\begin{lemma}\label{lem:CE.2}
    $\bar{\alpha}(L_m,G_r)=(m^2+m)\delta_{m+r,0}c_4$, where $c_4\in\bC$ and $m\in\bZ$, $r\in\bZ+\epsilon$.
\end{lemma}
\begin{proof}
    By \eqref{2-cocycle.Jacobi}, considering $\alpha(L_m,[L_n,G_r])=\alpha([L_m,L_n],G_r)+\alpha(L_n,[L_m,G_r])$, we obtain
    \begin{align}\label{alpha.LLG}
        r\alpha(L_m,G_{n+r})=(n-m)\alpha(L_{m+n},G_r)+r\alpha(L_n,G_{m+r}).
    \end{align}
    Letting $m=0$ in \eqref{alpha.LLG}, we have $r\alpha(L_0,G_{n+r})=(n+r)\alpha(L_n,G_r)$. Then, 
    \begin{equation*}
        \alpha(L_n,G_r)=\frac{r}{n+r}\alpha(L_0,G_{n+r}),\ \ n+r\neq 0.
    \end{equation*}
    It follows that 
    \begin{align*}
        \bar{\alpha}(L_n,G_r)&=\alpha(L_n,G_r)-\alpha_f(L_n,G_r)=\frac{r}{n+r}\alpha(L_0,G_{n+r})-f([L_n,G_r])\\
        &=\frac{r}{n+r}\alpha(L_0,G_{n+r})-rf(G_{n+r})=0,\ \  n+r\neq 0.
    \end{align*}
    The equality \eqref{alpha.LLG} also holds if $\alpha$ is replaced by $\bar{\alpha}$. Putting $r=-m-n$ in \eqref{alpha.LLG}, we can deduce that 
    \begin{equation}\label{alpha.LLG.1}
        (m+n)\bar{\alpha}(L_m,G_{-m})=(m-n)\bar{\alpha}(L_{m+n},G_{-m-n})+(m+n)\bar{\alpha}(L_n,G_{-n}).
    \end{equation}
    Suppose that $\bar{\alpha}(L_m,G_{-m})=g(m)$. Then \eqref{alpha.LLG.1} turns into 
    \begin{equation}\label{alpha.LLG.2}
        (m+n)g(m)=(m-n)g(m+n)+(m+n)g(n).
    \end{equation}
    Putting $n=-m$ in \eqref{alpha.LLG.2}, we get $2mg(0)=0$, which implies that $g(0)=0$. Define $f(G_0)=\alpha(L_{-1},G_1)$. Then, 
    \begin{align*}
        g(-1)&=\bar{\alpha}(L_{-1},G_1)=\alpha(L_{-1},G_1)-\alpha_f(L_{-1},G_1)\\
    &=\alpha(L_{-1},G_1)-f([L_{-1},G_1])=\alpha(L_{-1},G_1)-f(G_0)=0.
    \end{align*}
    Taking $n=1$ in \eqref{alpha.LLG.2}, we have
    \begin{equation}\label{alpha.LLG.3}
        (m+1)g(m)=(m-1)g(m+1)+(m+1)g(1).
    \end{equation}
    Moreover, setting $n=-1$ and replacing $m$ by $m+1$ in \eqref{alpha.LLG.2}, we can derive that
    \begin{equation}\label{alpha.LLG.4}
        mg(m+1)=(m+2)g(m).
    \end{equation}
    Combining \eqref{alpha.LLG.3} and \eqref{alpha.LLG.4}, we get $g(m)=\frac{1}{2}(m^2+m)g(1)$ for any $m\in\bZ$. Therefore, $$\bar{\alpha}(L_m,G_r)=(m^2+m)\delta_{m+r,0}c,$$ for some constant $c\in\bC$.
\end{proof}

\begin{lemma}\label{lem:CE.3}
    $\bar{\alpha}(W_m,G_r)=\delta_{m+r,0}c_5$, where $c_5\in\bC$ and $m\in\bZ$, $r\in\bZ+\epsilon$.
\end{lemma}
\begin{proof}
    By \eqref{2-cocycle.Jacobi}, considering $\bar{\alpha}(L_m,[W_n,G_r])=\bar{\alpha}([L_m,W_n],G_r)+\bar{\alpha}(W_n,[L_m,G_r])$, we obtain
    \begin{equation}\label{alpha.LWG}
        (m+n)\bar{\alpha}(W_{m+n},G_r)+r\bar{\alpha}(W_n,G_{m+r})=0.
    \end{equation}
    Putting $m=0$ in \eqref{alpha.LWG}, we get $(n+r)\bar{\alpha}(W_n,G_r)=0$, which implies that $\bar{\alpha}(W_n,G_r)=0$ if $n+r\neq 0$. Thus for $\epsilon=\frac{1}{2}$, we have $\bar{\alpha}(W_m,G_r)=0$ for any $m\in\bZ,r\in\bZ+\frac{1}{2}$. If $\epsilon=0$, we can take $n=0$ and $r=-m$ in \eqref{alpha.LWG} and get $m\bar{\alpha}(W_m,G_{-m})=m\bar{\alpha}(W_0,G_0)$, which leads to $\bar{\alpha}(W_m,G_{-m})=\bar{\alpha}(W_0,G_0)=c$, where $c\in\bC$ is a constant.
\end{proof}

\begin{lemma}\label{lem:CE.4}
    $c_2=0$ and $\bar{\alpha}(G_r,G_s)=\delta_{r+s,0}c_3$, where $r,s\in\bZ+\epsilon$.
\end{lemma}
\begin{proof}
    By \eqref{2-cocycle.Jacobi}, using $\bar{\alpha}(L_m,[G_r,G_s])=\bar{\alpha}([L_m,G_r],G_s)+\bar{\alpha}(G_r,[L_m,G_s])$, we conclude that
    \begin{equation}\label{alpha.LGG}
        (r+s)\bar{\alpha}(L_m,W_{r+s})=r\bar{\alpha}(G_{m+r},G_s)+s\bar{\alpha}(G_r,G_{m+s}).
    \end{equation}
    Taking $m=0$ in \eqref{alpha.LGG}, we have $(r+s)\bar{\alpha}(L_0,W_{r+s})=(r+s)\bar{\alpha}(G_r,G_s)$. It implies that
    \begin{equation}\label{alpha.LGG.1}
        \bar{\alpha}(L_0,W_{r+s})=\bar{\alpha}(G_r,G_s),\ \ r+s\neq 0.
    \end{equation}
    By Lemma~\ref{lem:CE.1}, we know that $\bar{\alpha}(L_0,W_{r+s})=0$ if $r+s\neq 0$. This together with \eqref{alpha.LGG.1} shows that $\bar{\alpha}(G_r,G_s)=0$ if $r+s\neq 0$.

    Putting $m=-r-s$ in \eqref{alpha.LGG} gives
    \begin{equation}\label{alpha.LGG.2}
        (r+s)\bar{\alpha}(L_{-r-s},W_{r+s})=r\bar{\alpha}(G_{-s},G_s)+s\bar{\alpha}(G_r,G_{-r}).
    \end{equation}
    Note that $\bar{\alpha}(G_{-s},G_s)=-(-1)^{|G_{-s}||G_s|}\bar{\alpha}(G_s,G_{-s})=\bar{\alpha}(G_s,G_{-s})$ by \eqref{2-cocycle.skew.symm}. Suppose that $\bar{\alpha}(G_s,G_{-s})=g(s)$. Then \eqref{alpha.LGG.2} becomes 
    \begin{equation}\label{alpha.LGG.3}
        (r+s)\bar{\alpha}(L_{-r-s},W_{r+s})=rg(s)+sg(r).
    \end{equation}
    Taking $r=-r$, $s=-s$ in \eqref{alpha.LGG.3}, we have
    \begin{equation}\label{alpha.LGG.4}
        (r+s)\bar{\alpha}(L_{r+s},W_{-r-s})=rg(s)+sg(r).
    \end{equation}\label{alpha.LGG.5}
    From \eqref{alpha.LGG.3} and \eqref{alpha.LGG.4}, we can deduce that
    \begin{equation*}
        (r+s)\bar{\alpha}(L_{-r-s},W_{r+s})=(r+s)\bar{\alpha}(L_{r+s},W_{-r-s}).
    \end{equation*}
   By Lemma~\ref{lem:CE.1}, we get $2(r+s)^2c_2=0$, which leads to $c_2=0$. So \eqref{alpha.LGG.3} turns into 
   \begin{equation}\label{alpha.LGG.2.1}
       (r+s)c_3=rg(s)+sg(r).
   \end{equation}
   
   If $r,s\in\bZ$, then we can take $s=1$ in \eqref{alpha.LGG.2.1} and obtain
   \begin{equation}\label{alpha.LGG.2.2}
       (r+1)c_3=rg(1)+g(r).
   \end{equation}
   Putting $r=1$ in \eqref{alpha.LGG.2.2} gives $g(1)=c_3$. Plugging it into \eqref{alpha.LGG.2.2}, we have $g(r)=c_3$ for $r\in\bZ$.

   If $r,s\in\bZ+\frac{1}{2}$, then we take $s=\frac{1}{2}$ in \eqref{alpha.LGG.2.1} and obtain 
   \begin{equation}\label{alpha.LGG.2.3}
       (r+\frac{1}{2})c_3=rg(\frac{1}{2})+\frac{1}{2}g(r).
   \end{equation}
   We can get $g(\frac{1}{2})=c_3$ by taking $r=\frac{1}{2}$ in \eqref{alpha.LGG.2.3}. Hence, $g(r)=c_3$ for $r\in\bZ+\frac{1}{2}$, and the lemma follows.
\end{proof}

By Lemmas~\ref{lem:CE.1}--\ref{lem:CE.4}, we get the central extension of the Lie superalgebras $\bar{\cS}^\epsilon$. The new algebras denoted by $\tilde{\cS}^\epsilon$ have a basis $\{L_m, W_m, G_r, C_1, C_2, C_3, C_4\ |\ m\in\bZ,\ r\in\bZ+\epsilon\}$, with super-brackets:
\begin{align*}
    &[L_m,L_n]=(n-m)L_{m+n}+\frac{m^3-m}{12}\delta_{m+n,0}C_1,\\
    &[L_m,W_n]=(m+n)W_{m+n}+\delta_{m+n,0}C_2,\\
    &[L_m,G_r]=rG_{m+r}+(m^2+m)\delta_{m+r,0}C_3,\\
    &[W_m,G_r]=\delta_{m+r,0}C_4,\\
    &[G_r,G_s]=(r+s)W_{r+s}+\delta_{r+s,0}C_2,\\
    &[W_m,W_n]=[\mathcal{S}^\epsilon,C_i]=0,\ \ i=1,2,3,4,
\end{align*}
for any $m,n\in\bZ$, $r,s\in\bZ+\epsilon$, $\epsilon=\frac{1}{2},0$. Note that $\bC C_3$ and $\bC C_4$ are odd central elements.

In addition, by Proposition~\ref{prop:2.coho}, we arrive at the following theorem.
\begin{theorem}\label{thm:2-cohology}
    \subno{1} $\dim H^2(\cS^0,\bC)=4$\textup{;} $H^2(\cS^0,\bC)=\bC\alpha_1\oplus\bC\alpha_2\oplus\bC\alpha_3\oplus\bC\alpha_4$, where
    \begin{align*}
        &\alpha_1(L_m,L_n)=\frac{m^3-m}{12}\delta_{m+n,0},\\
        &\alpha_2(L_m,W_n)=\delta_{m+n,0},\ \ \alpha_2(G_r,G_s)=\delta_{r+s,0},\\
        &\alpha_3(L_m,G_r)=(m^2+m)\delta_{m+r,0},\\
        &\alpha_4(W_m,G_r)=\delta_{m+r,0},
    \end{align*}
    for any $m,n,r,s\in\bZ$.

    \subno{2} $\dim H^2(\cS^\frac{1}{2},\bC)=2$\textup{;} $H^2(\cS^\frac{1}{2},\bC)=\bC\beta_1\oplus\bC\beta_2$, where
    \begin{align*}
        &\beta_1(L_m,L_n)=\frac{m^3-m}{12}\delta_{m+n,0},\\
        &\beta_2(L_m,W_n)=\delta_{m+n,0},\ \ \beta_2(G_r,G_s)=\delta_{r+s,0},
    \end{align*}
    for any $m,n\in\bZ$, $r,s\in\bZ+\frac{1}{2}$.
\end{theorem}

\subsection{Lie superalgebra $\cS$}
\hspace{1.5em}In the following, we only focus on the case of the central extension of $\bar{\cS}^\frac{1}{2}$, that is the Lie superalgebra $\cS$ defined in Definition~\ref{def:e}.

First, recall that the Ovsienko--Roger algebra $\mathcal{L}_\lambda$ is the centrally extended Lie algebra of $\mathcal{W}(a,b)$, which arose from the study on matrix analogues of the Sturm--Liouville operators by Ovsienko and Roger (cf. \cite{MOG}). For the case of $\lambda=1$, we have the following definition.
\begin{definition}
    The $\lambda=1$ Ovsienko--Roger algebra
    \begin{equation*}
        \mathcal{L}_1=\bigoplus_{n\in\bZ}\bC L_n\oplus\bigoplus_{n\in\bZ}\bC W_n\oplus\bC C_1\oplus\bC C_2\oplus\bC C_3
    \end{equation*}
    is a Lie algebra with the following brackets ($m,n\in\bZ$):
    \begin{align*}
        &[L_m,L_n]=(n-m)L_{m+n}+\frac{m^3-m}{12}\delta_{m+n,0} C_1,\\
        &[L_m,W_n]=(m+n)W_{m+n}+\delta_{m+n,0}(m C_2+ C_3),\\
        &[W_m,W_n]=[\mathcal{L}_1,C_i]=0,\ \ i=1,2,3.
    \end{align*}
\end{definition}
We say a Lie algebra $\fg$ is \emph{perfect} if $\fg=[\fg,\fg]$. Obviously, $\mathcal{L}_1$ is non-perfect since $W_0\notin[\mathcal{L}_1,\mathcal{L}_1]$. The category of restricted modules for $\mathcal{L}_1$ was studied in \cite{LPX22}. Also in \cite{LPX22}, many simple $\mathcal{L}_1$-modules were constructed. 

Now, we return to the Lie superalgebra $\cS$. Clearly, the even subalgebra of $\cS$ is a subalgebra of $\mathcal{L}_1$. The lack of $C_2$ (of $\mathcal{L}_1$) leads to significant differences in their modules. We will show it later. Moreover, $\cS$ is still non-perfect due to $W_0$.

By Definition~\ref{def:e}, we have the decomposition $\cS=\cS_\ep\oplus\cS_\op$, where
\begin{equation*}
    \cS_\ep=\bigoplus_{m\in\bZ}(\bC L_m\oplus\bC W_m)\oplus\bC C_1\oplus\bC C_2,\ \ \cS_\op=\bigoplus_{r\in\bZ+\frac{1}{2}}\bC G_r.
\end{equation*}
It follows that $\cS$ possesses a triangular decomposition $\cS=\cS_-\oplus\cS_0\oplus\cS_+$, where
\begin{align*}
    &\cS_{\pm}=\bigoplus_{m\in\bZ_+}(\bC L_{\pm m}\oplus\bC W_{\pm m})\oplus\bigoplus_{r\in\bN+\frac{1}{2}}\bC G_{\pm r},\\
    &\cS_0=\bC L_0\oplus\bC W_0\oplus\bC C_1\oplus\bC C_2.
\end{align*}

Let $\varphi:\cS_0\rightarrow\bC$ be a linear function. Suppose that $V$ is an $\cS_0$-module with the action of the center $z=\bC C_1\oplus\bC C_2$ on $V$ is given by $\varphi$ and $\cS_+$ acts trivially on $V$, making $V$ an $(\cS_0\oplus\cS_+)$-module. The \emph{generalized Verma module} over $\cS$ can be defined by
\begin{equation*}
    M_\varphi(V)=\cU(\cS)\otimes_{\cU(\cS_0\oplus\cS_+)}V.
\end{equation*}
In particular, for $h_1,h_2,c_1\in\bC$, let $V=\bC v$ be a one-dimensional $(\cS_0\oplus\cS_+)$-module defined by
\begin{equation*}
    L_0v=h_1 v,\ \ W_0 v=h_2 v,\ \ C_1v=c_1v,\ \ C_2v=0,\ \ \cS_+v=0,
\end{equation*}
where $\varphi(L_0)=h_1$, $\varphi(W_0)=h_2$, $\varphi(C_1)=c_1$ and $\varphi(C_2)=0$. Note that $\varphi(C_2)=0$ is necessary since $[L_0,W_0]=C_2$. Then the \emph{Verma module} of $\cS$ can be defined by 
\begin{equation*}
    M(h_1,h_2,c_1)=\cU(\cS)\otimes_{\cU(\cS_0\oplus\cS_+)}\bC v.
\end{equation*}

For any $k\in\bZ_+$, let
\begin{equation*}
    \cS^{(k)}=\bigoplus_{i\ge k}(\bC L_i\oplus\bC W_i\oplus\bC G_{i+\frac{1}{2}})\oplus\bC C_1\oplus\bC C_2.
\end{equation*}
Then $\cS^{(k)}$ is a subalgebra of $\cS$. Let $\psi_k:\cS^{(k)}\rightarrow \bC$ be a Lie superalgebra homomorphism. It follows immediately from the relation \eqref{e.rel} that 
\begin{equation*}
        \psi_k(L_m)=0,\ \ \psi_k(W_n)=0,\ \ \psi_k(G_{r+\frac{1}{2}})=0,\ \ \forall\ m\ge 2k+1,\ n\ge 2k,\ r\ge k.
\end{equation*}
\begin{definition}
    For $c_1,c_2\in\bC$, an $\cS$-module $W$ is called a \emph{Whittaker module} of type $(\psi_k,c_1,c_2)$ if

    \subno{1} $W$ is generated by a homogeneous vector $w$,

    \subno{2} $xw=\psi_k(x)w$ for any $x\in\cS^{(k)}$,

    \subno{3} $C_1w=c_1w$, $C_2w=c_2w$,\newline
    where $w$ is called a \emph{Whittaker vector} of $W$. 
\end{definition}

\begin{definition}
    An $\cS$-module $M$ is called \emph{restricted} in the sense that for every $v\in M$,  
    \begin{equation*}
        L_i v=W_i v=G_{i+\frac{1}{2}}v=0,
    \end{equation*}
    for $i$ sufficiently large.
\end{definition}
It is clear that both generalized Verma modules and Whittaker modules over $\cS$ are restricted.

Denote by $\bM$ the set of all infinite vectors of the form $\bi:=(\dots,i_2,i_1)$ with entries in $\bN$ such that the number of nonzero entries is finite and $\bM_1=\{\bi\in\bM\ |\ i_k=0,1,\ \forall\ k\in\bZ_+\}$. Let $\bO=(\dots,0,0)\in\bM$ (or $\bM_1$). For $i\in\bZ_+$, denote $e_i=(\dots,0,1,0,\dots,0)\in\bM$ (or $\bM_1$), where $1$ is in the $i$'th position from the right. For any $\bi\in\bM$ (or $\bM_1$), we write
\begin{align*}
    \bw(\bi)=\sum_{s\in\bZ_+}s\cdot i_s,\ \ \bd(\bi)=\sum_{s\in\bZ_+}i_s,
\end{align*}
which are nonnegative integers. For any nonzero $\bi\in\bM$ (or $\bM_1$), namely $\bi\neq\bO$, assume that $p$ and $q$ are the largest and smallest integers such that $i_p\neq 0$ and $i_q\neq 0$, respectively, and define $\bi'=\bi-e_p$ and $\bi''=\bi-e_q$.

\begin{definition}
    Denote by $>$ and $\succ$ the \emph{lexicographical total order} and the \emph{reverse lexicographical total order} on $\bM$ (or $\bM_1$) respectively, defined as follows:

    \subno{1} for any $\bi,\bj\in\bM$ (or $\bM_1$),
    \begin{equation*}
        \bj>\bi\ \Longleftrightarrow\ \text{there exists }r\in\bZ_+ \text{ such that }(j_s=i_s,\ \forall\ s> r) \text{ and }j_r>i_r;
    \end{equation*}

    \subno{2} for any $\bi,\bj\in\bM$ (or $\bM_1$),
    \begin{equation*}
        \bj\succ\bi\ \Longleftrightarrow\ \text{there exists }r\in\bZ_+ \text{ such that }(j_s=i_s,\ \forall\ 1\le s<r) \text{ and }j_r>i_r.
    \end{equation*}
\end{definition}

Obviously, we can extend the above total order to \emph{principal total order} on $\bM\times\bM_1\times\bM$, still denoted by $\succ$:
\begin{align*}
    (\bi_1,\bj_1,\bk_1)\succ(\bi_2,\bj_2,\bk_2)\Longleftrightarrow &\ (\bk_1,\bw(\bk_1))\succ(\bk_2,\bw(\bk_2))\ \text{ or }\\
    &\ \bk_1=\bk_2\ \text{ and }\ (\bj_1,\bw(\bj_1))\succ(\bj_2,\bw(\bj_2))\ \text{ or }\\
    &\ \bk_1=\bk_2,\ \bj_1=\bj_2\ \text{ and }\ \bi_1>\bi_2,
\end{align*}
for any $(\bi_1,\bj_1,\bk_1), (\bi_2,\bj_2,\bk_2)\in\bM\times\bM_1\times\bM$. Denote $\bM\times\bM_1\times\bM$ by $\bar{\bM}$.

\section{Construction of simple restricted $\cS$-modules}\label{sect:construction}
\hspace{1.5em}In this section, we present the construction of simple restricted modules over the Lie superalgebra $\cS$. These modules are obtained from simple modules over certain subalgebras of $\cS$.

For convenience, define the following subalgebra of $\cS$,
\begin{align*}
    \cT_{d}=\bigoplus_{i\ge 0}(\bC L_i\oplus\bC W_{i-d})\oplus\bigoplus_{i\ge 1}\bC G_{i-\frac{1}{2}}\oplus\bC C_1\oplus\bC C_2,
\end{align*}
where $d\in\bN$.

Letting $V$ be a simple $\cT_{d}$-module, we have the induced $\cS$-module
\begin{equation*}
    \Ind(V)=\cU(\cS)\otimes_{\cU(\cT_{d})}V.
\end{equation*}
Since we usually consider simple modules for the algebra $\cS$ or one of its subalgebras containing the central elements $C_i$, $i=1,2$, we always suppose that the action of $C_i$ is the scalar $c_i$ for $i=1,2$.

Fix $d\in\bN$ and let $V$ be a simple $\cT_{d}$-module. For $\bi,\bk\in\bM$, $\bj\in\bM_1$, we denote
\begin{equation*}
    W^\bi G^\bj L^\bk=\cdots W_{-d-2}^{i_2}W_{-d-1}^{i_1}\cdots G_{-\frac{3}{2}}^{j_2}G_{-\frac{1}{2}}^{j_1}\cdots L_{-2}^{k_2}L_{-1}^{k_1}\in \cU(\cS).
\end{equation*}
According to the PBW Theorem and $G_{i-\frac{1}{2}}^2V=(i-\frac{1}{2})W_{2i-1}V$ for $i\in\bZ$, every element of $\Ind(V)$ can be uniquely written in the following form
\begin{equation}\label{PBW}
    \sum_{(\bi,\bj,\bk)\in\bar{\bM}}W^\bi G^\bj L^\bk\vijk,
\end{equation}
where all $\vijk\in V$ and only finitely many of them are nonzero. For any $v\in\Ind(V)$ as in \eqref{PBW}, we denote by $\supp(v)$ the set of all $(\bi,\bj,\bk)\in\bar{\bM}$ such that $\vijk\neq 0$. For a nonzero $v\in\Ind(V)$, we write $\deg(v)$ the maximal element in $\supp(v)$ with respect to the principal total order on $\bar{\bM}$, which is called the \emph{degree} of $v$. Note that here and later we make the convention that $\deg(v)$ is defined only for $v\neq 0$.

\begin{theorem}\label{thm:simple.restricted.module}
    Let $d\in\bN$ and $V$ be a simple $\cT_{d}$-module and suppose there exists $t\in\bN$ satisfying the following two conditions\textup{:}
    
    \subno{a} the action of $W_t$ on $V$ is injective if $t>0$ or $c_2\neq 0$ if $t=0$\textup{;}

    \subno{b} $W_i V=0$ for all $i>t$ and $L_j V=0$ for all $j>t+d$.
    \newline
    Then we have

    \subno{i} $G_{i-\frac{1}{2}}V=0$ for all $i>t$\textup{;}

    \subno{ii} $\Ind(V)$ is a simple $\cS$-module.
\end{theorem}
\begin{proof} (i) For $i>t\ge 0$, we have $G_{i-\frac{1}{2}}^2V=(i-\frac{1}{2})W_{2i-1}V=0$ by $[G_{i-\frac{1}{2}},G_{i-\frac{1}{2}}]=(2i-1)W_{2i-1}$ and (b). If $G_{i-\frac{1}{2}}V=0$, then we are done. Otherwise, we can consider the proper subspace $U=G_{i-\frac{1}{2}}V$ of $V$ and deduce that
\begin{align*}
    &G_{j-\frac{1}{2}}U=G_{j-\frac{1}{2}}G_{i-\frac{1}{2}}V=(i+j-1)W_{i+j-1}V-G_{i-\frac{1}{2}}G_{j-\frac{1}{2}}V\subseteq G_{i-\frac{1}{2}}V=U,\ \ \forall\ j\in\bZ_+,\\
    &L_m U=L_m G_{i-\frac{1}{2}}V=(i-\frac{1}{2})G_{m+i-\frac{1}{2}}V+G_{i-\frac{1}{2}}L_mV\subseteq U,\ \ \forall\ m\in\bN,\\
    &W_nU=W_n G_{i-\frac{1}{2}}V=G_{i-\frac{1}{2}}W_nV\subseteq U,\ \ \forall\ n\in\bZ,\ n\ge -d.
\end{align*}
It follows that $U$ is a proper $\cT_d$-submodule of $V$. Then $U=G_{i-\frac{1}{2}}V=0$ for $i> t$ since $V$ is simple.

    (ii) In order to prove this conclusion, we need the following claim.
    \begin{claim}
        For any $v\in\Ind(V)\setminus V$, let $\deg(v)=(\bi,\bj,\bk)$, $\check{i}=\max\{s:i_s\neq 0\}$ if $\bi\neq\bO$, $\hat{j}=\min\{s:j_s\neq 0\}$ if $\bj\neq\bO$ and $\hk=\min\{s:k_s\neq 0\}$ if $\bk\neq\bO$. Then we can obtain

        \subno{1} if $\bk\neq\bO$, then $\hk>0$ and $\deg(W_{\hk+t}v)=(\bi,\bj,\bk'')$\textup{;}

        \subno{2} if $\bk=\bO,\bj\neq\bO$, then $\hj>0$ and $\deg(G_{\hj+t-\frac{1}{2}}v)=(\bi,\bj'',\bO)$\textup{;}

        \subno{3} if $\bj=\bk=\bO,\bi\neq\bO$, then $\ci>0$ and $\deg(L_{\ci+t+d}v)=(\bi',\bO,\bO)$.
    \end{claim}

    To prove this, we assume that $v$ is of the form in \eqref{PBW}. It is enough to show what we want by comparing the degree.

    \subno{1} Consider those $\vxyz$ with $W_{\hk+t}W^\bx G^\by L^\bz\vxyz\neq 0$. Note that $W_{\hk+t}\vxyz=0$ for any $(\bx,\by,\bz)\in\supp(v)$. It is easy to check that 
    \begin{equation*}
        W_{\hk+t}W^\bx G^\by L^\bz\vxyz=W^\bx G^\by [W_{\hk+t},L^\bz]\vxyz.
    \end{equation*}
    Clearly, $W_t\vxyz\neq 0$ or $c_2\vxyz\neq0$ by (a). If $\bz=\bk$, one can get that
    \begin{equation*}
        \deg(W_{\hk+t}W^\bx G^\by L^\bz\vxyz)=(\bx,\by,\bz'')\preceq(\bi,\bj,\bk''),
    \end{equation*}
    where the equality holds if and only if $\by=\bj,\bx=\bi$.

    Now, we assume that $(\bz,\bw(\bz))\prec(\bk,\bw(\bk))$ and denote
    \begin{equation*}
        \deg(W_{\hk+t}W^\bx G^\by L^\bz\vxyz)=(\bx_1,\by_1,\bz_1)\in\bar{\bM}.
    \end{equation*}
    If $\bw(\bz)<\bw(\bk)$, then we obtain $\bw(\bz_1)\le\bw(\bz)-\hk<\bw(\bk)-\hk=\bw(\bk'')$, which leads to $(\bx_1,\by_1,\bz_1)\prec(\bi,\bj,\bk'')$. Then we discuss the case $\bw(\bz)=\bw(\bk)$ and $\bz\prec\bk$. Note that here $\bz\neq\bO$, otherwise $\bw(\bz)=0<\bw(\bk)$ since $\bk\neq\bO$. Let $\hat{z}=\min\{s:z_s\neq 0\}>0$. If $\hat{z}>\hat{k}$, then $\bw(\bz_1)<\bw(\bz)-\hk=\bw(\bk'')$. If $\hat{z}=\hk$, we can similarly deduce that $(\bx_1,\by_1,\bz_1)=(\bx,\by,\bz'')$. Then by $\bz''\prec\bk''$, we have 
    \begin{equation*}
        \deg(W_{\hk+t}W^\bx G^\by L^\bz\vxyz)=(\bx_1,\by_1,\bz_1)\prec(\bi,\bj,\bk'')
    \end{equation*}
    in both cases.

    Hence, by combining all the arguments above, we can conclude that $\deg(W_{\hk+t}v)=(\bi,\bj,\bk'')$ as desired.

    \subno{2} Now we consider $\vxy$ with $G_{\hj+t-\frac{1}{2}}W^\bx G^\by\vxy\neq 0$. Since $G_{\hj+t-\frac{1}{2}}\vxy=0$ for any $(\bx,\by,\bO)\in\supp(v)$, it follows that
    \begin{align*}
        G_{\hj+t-\frac{1}{2}}W^\bx G^\by\vxy=W^\bx [G_{\hj+t-\frac{1}{2}},G^{\by}]\vxy.
    \end{align*}
    Note that $W_t\vxy\neq 0$ or $c_2\vxy\neq0$ by (a). If $\by=\bj$, one can get that
    \begin{equation*}
        \deg(G_{\hj+t-\frac{1}{2}}W^\bx G^\by\vxy)=(\bx,\by'',\bO)\preceq(\bi,\bj'',\bO),
    \end{equation*}
    and the equality holds if and only if $\bx=\bi$.

    Now suppose that $(\by,\bw(\by))\prec(\bj,\bw(\bj))$ and denote 
    \begin{equation*}
        \deg(G_{\hj+t-\frac{1}{2}}W^\bx G^\by\vxy)=(\bx_1,\by_1,\bO)\in\bar{\bM}.
    \end{equation*}
    If $\bw(\by)<\bw(\bj)$, then we have $\bw(\by_1)\le\bw(\by)-\hj<\bw(\bj)-\hj=\bw(\bj'')$, which gives rise to $(\bx_1,\by_1,\bO)\prec(\bi,\bj'',\bO)$. Now we consider the case $\bw(\by)=\bw(\bj)$ and $\by\prec\bj$, which implies that $\by\neq\bO$, otherwise $\bw(\by)=0<\bw(\bj)$ since $\bj\neq\bO$. Let $\hat{y}=\min\{s:y_s\neq 0\}>0$. If $\hat{y}>\hat{j}$, then we have $\bw(\by_1)<\bw(\by)-\hj=\bw(\bj'')$. If $\hat{y}=\hj$, one can similarly check that $(\bx_1,\by_1,\bO)=(\bx,\by'',\bO)$. By the fact that $\by''\prec\bj''$, in both cases, we can get 
    \begin{equation*}
        \deg(G_{\hj+t-\frac{1}{2}}W^\bx G^\by\vxy)=(\bx_1,\by_1,\bO)\prec(\bi,\bj'',\bO).
    \end{equation*}

    Therefore, we can conclude that $\deg(G_{\hj+t-\frac{1}{2}}v)=(\bi,\bj'',\bO)$.

    \subno{3} Similarly, consider those $\vx$ with $L_{\ci+t+d}W^\bx\vx\neq 0$. Set $\check{x}=\max\{s:x_s\neq 0\}$, then $0<\check{x}\le\ci$. If $\check{x}<\ci$, then it is easy to obtain that $\deg(L_{\ci+t+d}W^\bx\vx)\prec(\bi',\bO,\bO)$. If $\check{x}=\ci$, we have
    \begin{align*}
        L_{\ci+t+d}W^\bx\vx&=L_{\check{x}+t+d}W^{x_{\check{x}}}_{-\check{x}-d} W_{-\check{x}+1-d}^{x_{\check{x}-1}}\cdots W_{-1-d}^{x_1}\vx\\
        &=[L_{\check{x}+t+d},W_{-\check{x}-d}^{x_{\check{x}}}]W_{-\check{x}+1-d}^{x_{\check{x}-1}}\cdots W_{-1-d}^{x_1}\vx,\\
        &=\begin{cases}
            tx_{\check{x}}W^{x_{\check{x}}-1}_{-\check{x}-d} W_{-\check{x}+1-d}^{x_{\check{x}-1}}\cdots W_{-1-d}^{x_1}W_t\vx,\ &t>0,\\
            c_2x_{\check{x}}W^{x_{\check{x}}-1}_{-\check{x}-d} W_{-\check{x}+1-d}^{x_{\check{x}-1}}\cdots W_{-1-d}^{x_1}\vx,\ &t=0.
        \end{cases}
    \end{align*}
     This together with (a) implies that
     \begin{equation*}
         \deg(L_{\ci+t+d}W^\bx\vx)\preceq(\bi',0,0),
     \end{equation*}
     where the equality holds if and only if $\bx=\bi$. As a result, Claim 1 has been proved. 

    Using Claim 1 repeatedly, from any nonzero element $v\in\Ind(V)$ we can reach a nonzero element in $\cU(\cS)v\cap V\neq 0$, which indicates the simplicity of $\Ind(V)$.
\end{proof}

\begin{remark}
    It is clear that the induced module $\Ind(V)$ is a simple restricted $\cS$-module.
\end{remark}

In addition, the above conclusions also hold if we do not assume the simplicity of $V$. We have the following corollary.
\begin{corollary}
    Letting $d,t$ and $V$ be as in Theorem~\ref{thm:simple.restricted.module} except that $V$ may not be simple as an $\cT_d$-module, then we have
    \begin{equation*}
        V=\{v\in\Ind(V)\ |\ W_i v=G_{j-\frac{1}{2}}v=L_k v=0,\ \forall\ i,j>t,\ k>t+d\}.
    \end{equation*}
\end{corollary}

\section{Characterization of simple restricted modules}\label{sect:char}
\hspace{1.5em}In this section, we give a precise characterization of simple restricted $\cS$-module under certain conditions. For any $r_1,r_2,r_3\in\bN$, set 
\begin{equation*}
    \cS^{(r_1,r_2,r_3)}=\bigoplus_{i\ge r_1}\bC L_i\oplus\bigoplus_{i\ge r_2}\bC W_i\oplus\bigoplus_{i\ge r_3}\bC(1-\delta_{i,0}) G_{i-\frac{1}{2}}.
\end{equation*}
It is straightforward to check that $\cS^{(r_1,r_2,r_3)}$ is a finitely generated subalgebra of $\cS$.

First, we show several equivalent conditions for simple restricted $\cS$-modules.

\begin{proposition}\label{prop:eq.cods}
    Let $P$ be a simple $\cS$-module. Then the following conditions are equivalent\textup{:}

    \subno{1} There exists $t\in\bN$ such that the actions of $L_i,W_i,(1-\delta_{i,0})G_{i-\frac{1}{2}}$, $i\ge t$ on $P$ are locally finite.

    \subno{2} There exists $t\in\bN$ such that the actions of $L_i,W_i,(1-\delta_{i,0})G_{i-\frac{1}{2}}$, $i\ge t$ on $P$ are locally nilpotent.

    \subno{3} There exist $r_1,r_2,r_3\in\bN$ such that $P$ is a locally finite $\cS^{(r_1,r_2,r_3)}$-module.

    \subno{4} There exist $r_1,r_2,r_3\in\bN$ such that $P$ is a locally nilpotent $\cS^{(r_1,r_2,r_3)}$-module.

    \subno{5} There exist $r_1,r_2,r_3\in\bN$ and a nonzero vector $v\in P$ such that $\cS^{(r_1,r_2,r_3)}v=0$.

    \subno{6} $P$ is restricted.
\end{proposition}
\begin{proof}
    Note that $(4)\Rightarrow(2)\Rightarrow(1)$, $(3)\Rightarrow(1)$ and $(6)\Rightarrow(5)$ are easy to get. Since $\cS^{(r_1,r_2,r_3)}$ is finitely generated, we have $(4)\Rightarrow(3)$. It is also not difficult to obtain $(5)\Rightarrow(4)$. Assume (5) holds. Since $P$ is a simple $\cS$-module, we have $P=\cU(\cS)v$ for the vector $v$ in (5). Then $P$ is a locally nilpotent $\cS^{(r'_1,r'_2,r'_3)}$-module by the PBW Theorem and the Lie super-brackets \eqref{e.rel}, for some $r'_1,r'_2,r'_3\in\bN$. Actually, we only need to show that $(1)\Rightarrow(5)\Rightarrow(6)$.
    
    Now we prove $(1)\Rightarrow(5)$. Suppose that $P$ is a simple $\cS$-module and there exists $t\in\bN$ such that the actions of $L_i,W_i,(1-\delta_{i,0})G_{i-\frac{1}{2}}$ are locally finite for all $i\ge t$. Then we can choose a nonzero $v\in P$ such that $L_t v=\lambda v$ for some $\lambda\in\bC$.

    Take any $j\in\bN$ with $j>t$ and denote
    \begin{align*}
        N_L&=\sum_{m\in\bN}\bC L_t^m L_j v=\cU(\bC L_t)L_j v,\\
        N_W&=\sum_{m\in\bN}\bC L_t^m W_j v=\cU(\bC L_t)W_j v,\\
        N_G&=\sum_{m\in\bN}\bC L_t^m G_{j-\frac{1}{2}} v=\cU(\bC L_t)G_{j-\frac{1}{2}} v,
    \end{align*}
    which are all finite-dimensional. By the relation \eqref{e.rel}, for all $m\in\bN$, we have
    \begin{align*}
        \big(j+(m-1)t\big)L_{j+(m+1)t}v=[L_t,L_{j+mt}]v=L_tL_{j+mt}v-L_{j+mt}L_tv=(L_t-\lambda)L_{j+mt}v.
    \end{align*}
    Similarly, we can obtain
    \begin{align*}
        \big(j+(m+1)t\big)W_{j+(m+1)t}v&=(L_t-\lambda)W_{j+mt}v,\\
        (j+mt-\frac{1}{2})G_{j+(m+1)t-\frac{1}{2}}v&=(L_t-\lambda)G_{j+mt-\frac{1}{2}}v,\ \ \forall\ m\in\bN.
    \end{align*}
    It follows that
    \begin{align*}
        &L_{j+mt}v\in N_L\Rightarrow L_{j+(m+1)t}v\in N_L,\ W_{j+mt}v\in N_W\Rightarrow W_{j+(m+1)t}v\in N_W,\\
        &G_{j+mt-\frac{1}{2}}v\in N_G\Rightarrow G_{j+(m+1)t-\frac{1}{2}}v\in N_G,\ \ \forall\ m\in\bN,\, j>t.
    \end{align*}
    Hence, by induction on $m$, we derive that
    \begin{equation*}
        L_{j+mt}v\in N_L,\ W_{j+mt}v\in N_W,\ G_{j+mt-\frac{1}{2}}v\in N_G,\ \ \forall\ m\in\bN.
    \end{equation*}
    In other words, we can conclude that $\sum_{m\in\bN}\bC L_{j+mt}v$, $\sum_{m\in\bN}\bC W_{j+mt}v$ and $\sum_{m\in\bN}\bC G_{j+mt-\frac{1}{2}}v$ are all finite-dimensional for $j>t$. Thus,
    \begin{align*}
        \sum_{i\in\bN}\bC L_{t+i}v&=\bC L_t v+\sum_{j=t+1}^{2t}(\sum_{m\in\bN}\bC L_{j+mt}v),\\
        \sum_{i\in\bN}\bC W_{t+i}v&=\bC W_t v+\sum_{j=t+1}^{2t}(\sum_{m\in\bN}\bC W_{j+mt}v),\\
        \sum_{i\in\bN}\bC G_{t+i-\frac{1}{2}}v&=\bC G_{t-\frac{1}{2}} v+\sum_{j=t+1}^{2t}(\sum_{m\in\bN}\bC G_{j+mt-\frac{1}{2}}v),
    \end{align*}
    are all finite-dimensional. Now, we can choose $l\in\bZ_+$ such that
    \begin{equation}\label{finite.sum}
        \sum_{i\in\bN}\bC L_{t+i}v=\sum_{i=0}^l\bC L_{t+i}v,\ \sum_{i\in\bN}\bC W_{t+i}v=\sum_{i=0}^l\bC W_{t+i}v,\ \sum_{i\in\bN}\bC G_{t+i-\frac{1}{2}}v=\sum_{i=0}^l\bC G_{t+i-\frac{1}{2}}v.
    \end{equation}
    Denote $V'=\sum_{x_0,\dots,x_l,z_0,\dots,z_l\in\bN,y_0,\dots,y_l\in\bY}\bC W_t^{x_0}\cdots W_{t+l}^{x_l}G_{t-\frac{1}{2}}^{y_0}\cdots G_{t+l-\frac{1}{2}}^{y_l}L_t^{z_0}\cdots L_{t+l}^{z_l}v$, where $\bY=\{0,1\}$. 

    \begin{claim}
        $V'$ is a finite-dimensional $\cS^{(t,t,t)}$-module.
    \end{claim}
    By (1), we know that $V'$ is finite-dimensional. Using the PBW Theorem, we can rewrite $L_{t+s}v'$, $W_{t+s}v'$, $G_{t+s-\frac{1}{2}}v'$ with $s\in\bN$ and $v'\in V'$ as a sum of vectors of the following form, respectively,
    \begin{subequations}\label{V'}
        \begin{align}
        &L_{t+s}W_t^{x_0}\cdots W_{t+l}^{x_l}G_{t-\frac{1}{2}}^{y_0}\cdots G_{t+l-\frac{1}{2}}^{y_l}L_t^{z_0}\cdots L_{t+l}^{z_l}v,\label{V'.L}\\
        &W_{t+s}W_t^{x_0}\cdots W_{t+l}^{x_l}G_{t-\frac{1}{2}}^{y_0}\cdots G_{t+l-\frac{1}{2}}^{y_l}L_t^{z_0}\cdots L_{t+l}^{z_l}v,\label{V'.W}\\
        &G_{t+s-\frac{1}{2}}W_t^{x_0}\cdots W_{t+l}^{x_l}G_{t-\frac{1}{2}}^{y_0}\cdots G_{t+l-\frac{1}{2}}^{y_l}L_t^{z_0}\cdots L_{t+l}^{z_l}v.\label{V'.G}
    \end{align}
    \end{subequations}
    By \eqref{finite.sum}, we only need to show that the elements in \eqref{V'.L}--\eqref{V'.G} with $0\le s\le l$ lie in $V'$. It is obvious for \eqref{V'.W}. For \eqref{V'.L}, we have
    \begin{align*}
        &\phantom{,=} L_{t+s}W_t^{x_0}\cdots W_{t+l}^{x_l}G_{t-\frac{1}{2}}^{y_0}\cdots G_{t+l-\frac{1}{2}}^{y_l}L_t^{z_0}\cdots L_{t+l}^{z_l}v\\
        &=W_t^{x_0}\cdots W_{t+l}^{x_l}G_{t-\frac{1}{2}}^{y_0}\cdots G_{t+l-\frac{1}{2}}^{y_l}L_t^{z_0}\cdots L_{t+s}^{z_s+1}\cdots L_{t+l}^{z_l}v\\
        &\quad+[L_{t+s},W_t^{x_0}\cdots W_{t+l}^{x_l}G_{t-\frac{1}{2}}^{y_0}\cdots G_{t+l-\frac{1}{2}}^{y_l}L_t^{z_0}\cdots L_{t+s}^{z_s}]L_{t+s+1}^{z_{s+1}}\cdots L_{t+l}^{z_l}v.
    \end{align*}
    By induction, one can easily check that all terms in above equality lie in $V'$. Hence, the elements of the form in \eqref{V'.L} lie in $V'$. The case of \eqref{V'.G} can be proved similarly. Therefore, we get Claim 2. 

    By Claim 2, we can take a minimal $n\in\bN$ such that
    \begin{equation}\label{finie.dim}
        (L_m+a_1L_{m+1}+\cdots+a_nL_{m+n})V'=0,
    \end{equation}
    for some $m\ge t$ and $a_i\in\bC$. Applying $L_m$ to \eqref{finie.dim}, we get
    \begin{equation*}
        (a_1[L_m,L_{m+1}]+\cdots+a_n[L_m,L_{m+n}])V'=0.
    \end{equation*}
    To avoid a contradiction, it forces that $n=0$, namely $L_m V'=0$. Hence, we obtain
    \begin{equation*}
        0=L_iL_mV'=[L_i,L_m]V'+L_m(L_iV')=(m-i)L_{m+i}V',\ \ \forall\ i\ge t.
    \end{equation*}
    So $L_{m+i}V'=0$ for all $i>m$. Then, for $j,k>m$, we have
    \begin{align*}
        0&=W_jL_mV'=[W_j,L_m]V'+L_m(W_jV')=-(m+j)W_{m+j}V',\\
        0&=G_{k-\frac{1}{2}}L_mV'=[G_{k-\frac{1}{2}},L_m]V'+L_m(G_{k-\frac{1}{2}}V')=(\frac{1}{2}-k)G_{m+k-\frac{1}{2}}V',
    \end{align*}
    that is, $W_{m+j}V'=G_{m+k-\frac{1}{2}}V'=0$. As a consequence, we get (5).

    For $(5)\Rightarrow(6)$, fix $r_1,r_2,r_3\in\bN$ and $0\neq v\in P$ such that $\cS^{(r_1,r_2,r_3)}v=0$. By the PBW Theorem and the simplicity of $P$, $P$ has a spanning set consisting of vectors of the form 
    \begin{equation*}
        L^\bi W^\bj G^\bk v=\cdots L_{r_1-2}^{i_2}L_{r_1-1}^{i_1}\cdots W_{r_2-2}^{j_2}W_{r_2-1}^{j_1}\cdots G_{r_3-\frac{3}{2}}^{k_2}G_{r_3-\frac{1}{2}}^{k_1}v,
    \end{equation*}
    where $\bi,\bj\in\bM$, $\bk\in\bM_1$.

    \begin{claim}
        There exists $n>N=r_1+r_2+r_3+\bw(\bi+\bj+\bk)$ such that
        \begin{equation*}
            L_nL^\bi W^\bj G^\bk v=W_nL^\bi W^\bj G^\bk v=G_{n-\frac{1}{2}}L^\bi W^\bj G^\bk v=0.
        \end{equation*}
    \end{claim}
    We prove the claim by induction on $D:=\bd(\bi+\bj+\bk)$. The case of $D=0$ is clear. For $D=1$, $L^\bi W^\bj G^\bk v$ is of the form $L_{r_1-b}v$, $W_{r_2-b}v$ or $G_{r_3+\frac{1}{2}-b}v$ for some $b\in\bZ_+$. Suppose $L^\bi W^\bj G^\bk v=G_{r_3+\frac{1}{2}-b}v$. For $n>r_1+r_2+r_3+b$, by assumption, we have
    \begin{align*}
        &G_{n-\frac{1}{2}}G_{r_3+\frac{1}{2}-b}v=[G_{n-\frac{1}{2}},G_{r_3+\frac{1}{2}-b}]v-G_{r_3+\frac{1}{2}-b}G_{n-\frac{1}{2}}v=0,\\
        &L_nG_{r_3+\frac{1}{2}-b}v=[L_n,G_{r_3+\frac{1}{2}-b}]v+G_{r_3+\frac{1}{2}-b}L_nv=0\ \text{ and }\ W_nG_{r_3+\frac{1}{2}-b}v=0.
    \end{align*}
    Similarly, we can obtain the results for the cases of $L_{r_1-b}v$ and $W_{r_2-b}v$. Now assume that $D>1$ and the claim holds for $D'<D$. Then for any $L^\bi W^\bj G^\bk v$ with $\bd(\bi+\bj+\bk)=D$, $n>N$, it follows that
    \begin{align*}
        L_nL^\bi W^\bj G^\bk v=&\sum_a\cdots L_{r_1-a-1}^{i_{a+1}}[L_n,L_{r_1-a}]L_{r_1-a}^{i_a-1}\cdots L_{r_1-1}^{i_1}W^\bj G^\bk v\\
        &+\sum_a L^\bi\cdots W_{r_2-a-1}^{j_{a+1}}[L_n,W_{r_2-a}]W_{r_2-a}^{j_a-1}\cdots W_{r_2-1}^{j_1}G^\bk v\\
        &+\sum_a L^\bi W^\bj\cdots G_{r_3-a-\frac{1}{2}}^{k_{a+1}}[L_n,G_{r_3-a+\frac{1}{2}}]G_{r_3-a+\frac{1}{2}}^{k_a-1}\cdots G_{r_3-\frac{1}{2}}^{k_1}v.
    \end{align*}
    By the induction hypothesis, we can get $L_nL^\bi W^\bj G^\bk v=0$ for $n>N$. Similarly, it can be verified directly that $W_nL^\bi W^\bj G^\bk v=G_{n-\frac{1}{2}}L^\bi W^\bj G^\bk v=0$ for $n>N$. Therefore, Claim 3 has been proved.

    By Claim 3, we can deduce that $P$ is a restricted module. This completes the proof. 
\end{proof}

Furthermore, we have the following results.
\begin{theorem}\label{thm:isom}
    Let $P$ be a simple restricted $\cS$-module. Suppose that there exists $c_2\neq 0$ such that $C_2 P=c_2 P$. Then there exists $d\in\bN$ and a simple $\cT_d$-module $V$ satisfying the conditions in Theorem~\ref{thm:simple.restricted.module} such that $P\cong\Ind(V)$.
\end{theorem}
\begin{proof}
    For any $\ti,\tj,\tk\in\bZ$, consider the vector space
    \begin{equation*}
        N_{\ti,\tj,\tk}=\{v\in P\ |\ W_iv=(1-\delta_{j,0})G_{j-\frac{1}{2}}v=L_kv=0,\ \ \forall\ i>\ti,j>\tj,k>\tk\}.
    \end{equation*}
    By Proposition~\ref{prop:eq.cods}(5), we know that $N_{\ti,\tj,\tk}\neq0$ for sufficiently large integers $\ti,\tj,\tk$. Note that $W_0v=c_2v\neq 0$ for any nonzero $v\in P$. Thus we can find a smallest nonnegative integer $a$, and choose some $b,c\in\bN$ with $b,c\ge a$ such that $N_{a,b,c}\neq 0$. Moreover, we can choose $a=b$ as in Theorem~\ref{thm:simple.restricted.module}.
    
    Denote $d=c-a$ and $V=N_{a,a,a+d}$. For any $i>a,j>a,k>a+d$, by \eqref{e.rel}, we have 
    \begin{align*}
        &W_i(G_{n-\frac{1}{2}}v)=0,\ \ G_{j-\frac{1}{2}}(G_{n-\frac{1}{2}}v)=(j+n-1)W_{j+n-1}v=0,\\ 
        &L_k(G_{n-\frac{1}{2}}v)=(n-\frac{1}{2})G_{k+n-\frac{1}{2}}v=0,
    \end{align*}
    for any $v\in V$, $n\ge 1$. Thus, $G_{n-\frac{1}{2}}v\in V$. Similarly, we can also get $W_{m-d}v\in V$ and $L_m v\in V$ for all $m\in\bN$. Therefore, $V$ is an $\cT_d$-module.

    By the definition of $V$, we can deduce that the action of $W_t$ on $V$ is injective, for some $t\in\bN$. Since $P$ is simple and generated by $V$, then there exists a canonical surjective map
    \begin{equation*}
        \pi:\Ind(V)\longrightarrow P,\ \ \pi(1\otimes v)=v,\ \ \forall\ v\in V.
    \end{equation*}
    To prove $\pi$ is bijective, we only need to show that it is injective. Let $K=\ker(\pi)$. Obviously, $K\cap V=0$. If $K\neq 0$, we can choose a nonzero vector $v\in K\setminus V$ such that $\deg(v)=(\bi,\bj,\bk)$ is minimal possible. Note that $K$ is an $\cS$-submodule of $\Ind(V)$ and hence is stable under the actions of $L_i,W_i$ and $G_{i-\frac{1}{2}}$ for all $i\in\bZ$. By the similar proof of Claim 1, we can create a new vector $u\in K$ with $\deg(u)\prec(\bi,\bj,\bk)$, which is a contradiction. Thus we have $K=0$, that is $P\cong\Ind(V)$. By the property of induced modules, we know $V$ is a simple $\cT_d$-module.
\end{proof}

As a mater of fact, any simple restricted $\cS$-module satisfying conditions in Theorem~\ref{thm:isom} is determined by some simple module $V$ over a certain subalgebra $\cT_d$. The conditions of Theorem~\ref{thm:simple.restricted.module} imply that $V$ can be viewed as a simple module over some finite-dimensional solvable quotient algebra of $\cT_d$. So the study of such modules over $\cS$ can be reduced to the study of simple modules over the corresponding finite-dimensional algebras. More precisely, we have the following conclusions.
\begin{theorem}\label{thm:quo.alg.mod}
    Let $P$ be a simple restricted $\cS$-module. Suppose that there exists $k\in\bZ_+$ such that the action of $W_k$ on $P$ is injective or there exists $c_2\neq 0$ such that $C_2 P=c_2 P$. Then $P\cong\Ind(V)$, where $V$ is a simple $\fq^{(d,t)}$-module, and $\mathfrak{q}^{(d,t)}=\cT_d/\cS^{(t+d+1,t+1,t+1)}$ is a quotient algebra for some $t\in\bN$.
\end{theorem}

However, the classification of simple $\fq^{(d,t)}$-modules remains an open problem, as far as we know, except when $(d,t)=(0,0)$. Note that $\fq^{(0,0)}$ is the Lie algebra $\cS_0$ and the classification of simple $\cS_0$-modules is due to an important work by R. Block \cite{Bl}. It follows from Theorem~\ref{thm:quo.alg.mod} that the simple generalized Verma modules for $\cS$ can be classified as follows.

\begin{corollary}\label{cor:gen.ver.mod}
    Every simple generalized Verma module for $\cS$ with $c_2\neq0$ is isomorphic to an induced module $\Ind(V)$, where $V$ is one of the $\cS_0$-modules including simple cuspidal modules, simple highest/lowest weight modules, simple (dual) Whittaker modules and simple Block modules.
\end{corollary}

Recall that the even subalgebra of $\cS$ is a subalgebra of the $\lambda=1$ Ovsienko--Roger algebra $\mathcal{L}_1$. Therefore, it is not difficult to write down the more explicit definitions of $\cS_0$-modules in Corollary~\ref{cor:gen.ver.mod} by referring to \cite[Corollary~4.4]{LPX22}. We leave them to interested readers. 

\section{Examples of restricted $\cS$-modules}\label{sect:example}
\hspace{1.5em}We are now in a position to give some examples of $\cT_d$-modules. Then, by Theorem~\ref{thm:simple.restricted.module}, we can construct (simple) $\cS$-modules.

\subsection{Verma module}
\hspace{1.5em}First, we give the example of the Verma module $M(h_1,h_2,c_1)$, which shows that it is always reducible. For $h_1,h_2,c_1\in\bC$, let $\bC v$ be a one-dimensional $\cS_0$-module defined by 
\begin{equation*}
    L_0v=h_1v,\ \ W_0v=h_2v,\ \ C_1v=c_1v,\ \ C_2v=0. 
\end{equation*}
Let $\cS_+$ act trivially on $v$, which makes $\bC v$ to be an $(\cS_0\oplus\cS_+)$-module. Then we get Verma module $ M(h_1,h_2,c_1)=\cU(\cS)\otimes_{\cU(\cS_0\oplus\cS_+)}\bC v$ of $\cS$. Consider the vector $G_{-\frac{1}{2}}v\in M(h_1,h_2,c_1)$, which is nonzero by the construction of the Verma module. Moreover, we have
\begin{align*}
    &L_m(G_{-\frac{1}{2}}v)=[L_m,G_{-\frac{1}{2}}]v+G_{-\frac{1}{2}}L_m v=0,\ \ \forall\ m>0,\\
    &W_m(G_{-\frac{1}{2}}v)=[W_m,G_{-\frac{1}{2}}]v+G_{-\frac{1}{2}}W_mv=0,\ \ \forall\ m>0,\\
    &G_r(G_{-\frac{1}{2}}v)=[G_r,G_{-\frac{1}{2}}]v-G_{-\frac{1}{2}}G_rv=0,\ \ \forall\ r\in\bN+\frac{1}{2},\\
    &L_0(G_{-\frac{1}{2}}v)=[L_0,G_{-\frac{1}{2}}]v+G_{-\frac{1}{2}}L_0 v=(h_1-\frac{1}{2})G_{-\frac{1}{2}}v,\\
    &W_0(G_{-\frac{1}{2}}v)=[W_0,G_{-\frac{1}{2}}]v+G_{-\frac{1}{2}}W_0v=h_2G_{-\frac{1}{2}}v.
\end{align*}
Hence, $G_{-\frac{1}{2}}v$ is a singular vector. Since the weight $h_1-\frac12$ of $G_{-\frac{1}{2}}v$ is strictly lower than the weight $h_1$ of $M(h_1,h_2,c_1)$, the submodule $\bC G_{-\frac{1}{2}}v$ is proper. Therefore, $M(h_1,h_2,c_1)$ is reducible. Note that $M(h_1,h_2,c_1)$ corresponds to the case $t=d=0$ in Theorem~\ref{thm:simple.restricted.module}, but it does not satisfy the condition $c_2\neq0$.

\begin{remark}
    This together with Corollary~\ref{cor:gen.ver.mod} describes completely the generalized Verma modules over $\cS$ with $c_2\neq 0$. Different from $\mathcal{L}_1$ in \cite[Corollary~4.4]{LPX22}, the Verma module of $\cS$ is always reducible.
\end{remark}

\subsection{Whittaker module}
\hspace{1.5em}For $k\in\bZ_+$, let $\psi_k$ be a Lie superalgebra homomorphism $\psi_k:\cS^{(k)}\rightarrow\bC$ with $\psi_k(C_1)=c_1$, $\psi_k(C_2)=c_2$. It follows that $\psi_k(L_m)=\psi_k(W_n)=\psi_k(G_{r+\frac{1}{2}})=0$, for $m\ge2k+1$, $n\ge2k$ and $r\ge k$. Let $\bC w$ be a one-dimensional vector space with
\begin{equation*}
    xw=\psi_k(x)w,\ \ C_1w=c_1w,\ \ C_2w=c_2w,\ \ \forall\ x\in\cS^{(k)}.
\end{equation*}
If $\psi_k(W_{2k-1})\neq0$, then $\bC w$ is a simple $\cS^{(k)}$-module. Now we set $\psi_k(W_{2k-1})\neq 0$ and consider the induced module
\begin{equation*}
   V_{\psi_k}=\cU(\cS_0\oplus\cS_+)\otimes_{\cU(\cS^{(k)})}\bC w.
\end{equation*}
It is easy to check that $V_{\psi_k}$ is a simple $(\cS_0\oplus\cS_+)$-module. Then the corresponding simple $\cS$-module $\Ind(V_{\psi_k})$ in Theorem~\ref{thm:simple.restricted.module} (here we take $t=2k-1$, $d=1$) is exactly the Whittaker module.

\end{document}